\documentclass[a4paper]{amsart}
\usepackage{latexsym,amsthm,amssymb}

\usepackage{bbm}

\DeclareMathAlphabet{\mathpzc}{OT1}{pzc}{m}{it}

\usepackage[leqno]{amsmath}

\usepackage{mathrsfs}

\usepackage[all]{xy}

\usepackage{color}
\usepackage{colordvi}

\usepackage{verbatim}
\usepackage[hyperindex=true,plainpages=false,colorlinks=true,]{hyperref}
\newtheorem{theorem}{Theorem}[section]

\newtheorem{definition}[theorem]{Definition}
\newtheorem{proposition}[theorem]{Proposition}

\newtheorem{remark}[theorem]{Remark}
\newtheorem{noname}[theorem]{}

\newtheorem*{acknowledgement}{Acknowledgements}

\newtheorem{lemma-conjecture}[theorem]{Lemma--Conjecture}

\newtheorem{corollary}[theorem]{Corollary}

\numberwithin{equation}{theorem}

\renewcommand{\mathcal}{\mathscr}

\newcommand{\SE}{{\mathcal{E}}}

\newcommand{\SI}{{\mathcal{I}}}

\newcommand{\SN}{{\mathcal{N}}}
\newcommand{\SO}{{\mathcal{O}}}

\newcommand{\ST}{{\mathcal{T}}}

\newcommand{\SY}{{\mathcal{Y}}}

\renewcommand{\mathbb}{\mathbf}

\newcommand{\PP}{\mathbb{P}}

\newcommand{\kk}{\mathbb{k}}

\newcommand{\WY}{\widetilde{Y}}

\title[An infinitesimal condition to deform a finite morphism to an embedding]{An infinitesimal condition to deform a finite morphism to an embedding}

\author{Francisco Javier Gallego}
\address{Departamento de \'Algebra, Universidad Complutense de Madrid}
\email{gallego@mat.ucm.es}

\author{Miguel Gonz\'alez}
\address{Departamento de \'Algebra, Universidad Complutense de Madrid}
\email{mgonza@mat.ucm.es}

\author{\\Bangere P. Purnaprajna}
\address{Department of Mathematics, University of Kansas}
\email{purna@math.ku.edu}

\subjclass[2000]{14B10, 14D15, 13D10, 14D06, 14A15, 14J10, 14J29}

\keywords{deformation of morphisms, embedding, smoothing, rope, canonical surface}
\thanks{The first and the second author were partially supported by grants MTM2006--04785 and MTM2009--06964 and by the UCM research group 910772. The first author also thanks the Department of Mathematics of the University of Kansas for its hospitality. The third author thanks the General Research Fund (GRF) of the University of Kansas for partially supporting this research. He also thanks the Algebra Department of the Universidad Complutense de Madrid for its hospitality.}
\begin{document}

\maketitle
\begin{abstract}
In this article we give a sufficient condition for a morphism $\varphi$ from a smooth variety $X$ to projective space, finite onto a smooth image, to be deformed to an embedding. This result puts some theorems on deformation of morphisms of curves and surfaces such as $K3$ and general type, obtained by ad hoc methods,  in a new, more conceptual light. One of the main interests of our result is to apply it to the construction of smooth varieties in projective space with given invariants. We illustrate this by using our result to construct canonically embedded surfaces with $c_1^2=3p_g-7$ and derive some interesting properties of their moduli spaces. Another interesting application of our result is the smoothing of ropes. We obtain a sufficient condition for a rope embedded in projective space to be smoothable. As a consequence, we prove that canonically embedded carpets satisfying certain conditions can be smoothed. We also give simple, unified proofs of known theorems on the smoothing of $1$--dimensional  ropes and $K3$ carpets. Our condition for deforming $\varphi$ to an embedding can be stated very transparently in terms of the cohomology class of a suitable first order infinitesimal deformation of $\varphi$. It holds in a very general setting (any $X$ of arbitrary dimension and any $\varphi$ unobstructed with an algebraic formally semiuniversal deformation). The simplicity of the result can be seen for instance when we specialize it to the case of curves.
\end{abstract}

\section*{Introduction}
The main purpose of this paper is to give a sufficient condition (see Theorem~\ref{cor.morphisms}) for a morphism $\varphi$ from a smooth variety $X$ to projective space, finite onto a smooth image, to be deformed to an embedding. This condition works in a very general setting, as it applies to any $X$ of arbitrary dimension and any $\varphi$ unobstructed having an algebraic formally semiuniversal deformation. The nature of the condition is infinitesimal. To the best of our knowledge, a general result of this nature, that is, an infinitesimal condition on $\varphi$ capable of assuring that $\varphi$ deforms not just to a morphism birational onto its image (compare with~\cite[Theorem 1.4]{canonical}) but to a closed immersion, was not previously known. 
Somehow surprisingly, looking at the cohomology class in $H^0(\mathcal N_\varphi)$ of a suitable first order infinitesimal deformation $\widetilde \varphi$ of $\varphi$ gives us substantial information on whether $\varphi$ can be deformed to an embedding. More precisely, put in a geometric language, the condition that we require $\widetilde \varphi$ to satisfy is for im$\widetilde \varphi$ to contain certain embedded \emph{rope} (see Definition~\ref{defrope}) sharing invariants with $X$.
One of the virtues of our result is its simplicity. This can be seen when we specialize it (see  Theorem~\ref{curves}), to the case of curves. In this regard, it is worth comparing the arguments of Section~\ref{intrinsic.section} with the proofs of \cite[Theorem 1.1]{GGP} or \cite[Theorem 3.2]{GGP2}. In these theorems the authors proved, by ad--hoc, more cumbersome methods, results regarding the deformation of finite morphisms to embeddings,  respectively in the case of curves and of $K3$ surfaces. In contrast, the arguments  used in Proposition~\ref{intrinsic} and Theorem~\ref{cor.morphisms} are more conceptual, elegant and transparent. Moreover,  Proposition~\ref{intrinsic} and Theorem~\ref{cor.morphisms} integrate results like \cite[Theorem 1.1]{GGP} and the results discussed in Section~\ref{canonical.surfaces.section} (see also~\cite[4.5]{AK}) into a more general, deeper understanding of the problem of deforming morphisms to embeddings.

\medskip

\noindent
In the remaining of the article we give several applications of the results  proved in Section~\ref{intrinsic.section}. In Section~\ref{ropes.section} we apply Proposition~\ref{intrinsic} and Theorem~\ref{cor.morphisms} to the problem of smoothing ropes. As mentioned in the previous paragraph, ropes appear naturally in the infinitesimal condition requested by Proposition~\ref{intrinsic} and Theorem~\ref{cor.morphisms}. This and the fact that Proposition~\ref{intrinsic} and Theorem~\ref{cor.morphisms} ``produce'' embeddings instead of just degree $1$ morphisms, allow us to use these two results to smooth ropes. In Theorem~\ref{cor.ropes} we find a necessary condition to smooth ropes of arbitrary dimension. As a corollary we obtain a result (see Proposition~\ref{canonical.carpets.smoothing}) on the smoothing of \emph{canonical carpets} (a canonical carpet is a rope of multiplicity $2$ that has the same invariants as a canonically embedded surface of general type). In Section~\ref{ropes.section}  we also unify, using the deformation of morphism approach set in Section~\ref{intrinsic.section}, Proposition~\ref{canonical.carpets.smoothing} and several previous results on smoothing of ropes. In this regard, we apply Theorem~\ref{cor.ropes} to ropes on curves and we get (see Theorem~\ref{curves.2}) a very simple, more conceptual proof of~\cite[Theorem 2.4]{GGP}. We also apply Proposition~\ref{intrinsic} to give a different proof of ~\cite[Theorem 3.5]{GP} on the smoothing of $K3$ \emph{carpets} (ropes of multiplicity $2$ with the invariants of a smooth $K3$ surface) on rational normal scrolls. We revisit as well~\cite[Theorem 3.5]{GGP2} on the smoothing of $K3$ carpets on Enriques surfaces.

\medskip
\noindent Another, very interesting use of Theorem~\ref{cor.morphisms} is the construction of smooth varieties with given invariants. The idea is to choose smooth finite covers $X$ of smooth, embedded, simpler varieties $Y$ so that $X$ have the desired invariants and then apply Theorem~\ref{cor.morphisms}. This process is illustrated in Section~\ref{canonical.surfaces.section}, where this method is used to construct canonically embedded surfaces with invariants on the Castelnuovo line $c_1^2=3p_g-7$, for any odd $p_g \geq 7$ (see Proposition~\ref{construct.canonical.surfaces}). The systematic construction of canonically embedded surfaces with given invariants is far from easy. The \linebreak
Castelnuovo line plays a particularly important role in the geography of surfaces of general type because surfaces with invariants  below the Castelnuovo line cannot be canonically embedded. The canonically embedded surfaces of Section~\ref{canonical.surfaces.section} appear as one deforms,  using Theorem~\ref{cor.morphisms}, canonical double covers of certain non--minimal rational surfaces. The construction was previously done in a  nicely geometric fashion by Ashikaga and Konno (see~\cite[4.5]{AK}).  Section~\ref{canonical.surfaces.section} puts Ashikaga and Konno's construction in the more general and conceptual light shed by Theorem~\ref{cor.morphisms}.

\medskip
\noindent We devote the last part of Section~\ref{canonical.surfaces.section} to study some properties of the moduli of the surfaces $X$ of Proposition~\ref{construct.canonical.surfaces}. The computations in the proof of Proposition~\ref{construct.canonical.surfaces} yield the dimension of the irreducible component $\mathcal M'$ of $[X]$ in its moduli space  (see Proposition~\ref{moduli} and Remark~\ref{twocomponents}). We prove as well that $\mathcal M'$ contains a codimension $1$ irreducible stratum through $[X]$, parameterizing surfaces of general type whose canonical map is a degree $2$ morphism. The existence of this ``hyperelliptic'' stratum of codimension $1$ has the following counterpart in terms of double structures: there is a unique canonically embedded carpet supported on $\varphi(X)$.  Finally we remark this interesting phenomenon: even though we are considering surfaces on the Castelnuovo line, for some values $(p_g, 0, c_1^2)$ their moduli space has not only components like $\mathcal M'$, whose general point corresponds to a surface which can be canonically embedded, but also components whose general point corresponds to a surface whose canonical map is a degree $2$ morphism (see Remark~\ref{twocomponents}). This bears evidence to the complexity of the moduli of surfaces of general type as compared for instance with the moduli of curves or the moduli of polarized $K3$ surfaces. In the latter two moduli spaces, except when the genus is $2$, all their components are alike in the sense that their general point corresponds to a non--hyperelliptic curve or $K3$ surface.

\section{A sufficient condition to deform finite morphisms to embeddings}\label{intrinsic.section}

The main purpose of this section is to give sufficient conditions (Proposition~\ref{intrinsic} and  Theorem~\ref{cor.morphisms}) for a first order deformation of a morphism from a smooth projective variety to projective space, finite onto a smooth image, to be extended to a family of embeddings. Before stating and proving these results we need to set up the notations and conventions used in this section and in the remaining of the article:

\medskip
\begin{noname}\label{setup}
{\bf  Notation and conventions.} {\rm Throughout this article we will use the following notation and conventions:
\begin{enumerate}
\item  We will work over an algebraically closed field $\mathbf k$ of characteristic $0$.
\item  $X$ and $Y$ will denote smooth, irreducible projective varieties.
\item  $i$ will denote a projective embedding $i: Y \hookrightarrow \mathbf P^N$. In this case, $\mathcal I$  will denote the ideal sheaf of $i(Y)$ in $\mathbf P^N$. Likewise, we will often abridge $i^{\ast}\mathcal O_{\mathbb P^N}(1)$ as $\mathcal O_Y(1)$.
\item $\pi$ will denote a finite morphism $\pi: X \longrightarrow Y$ of degree $n \geq 2$; in this case, $\mathcal E$ will denote  the trace--zero module of $\pi$ ($\mathcal E$ is a vector bundle on $Y$ of rank $n-1$).
\item $\varphi$ will denote a projective morphism $\varphi: X \longrightarrow \mathbf P^N$ such that $\varphi= i \circ \pi$.
\end{enumerate}}
\end{noname}

\noindent Recall that the normal sheaf of $\varphi$ is defined from the exact sequence
\begin{equation}\label{normal.sheaf}
 0 \longrightarrow \ST_X \longrightarrow \varphi^{\ast}\ST_{\PP^N} \longrightarrow \SN_{\varphi} \longrightarrow 0,
\end{equation}
injective at the left--hand--side by generic smoothness.
Recall also (see~\cite[Definition 3.4.6 and Lemma 3.4.7 (iii)]{Ser}) that $H^0(\mathcal N_\varphi)$ parametrizes first order infinitesimal deformations of $\varphi$.
In order to state and prove Proposition~\ref{intrinsic} and Theorem~\ref{cor.morphisms} we need to introduce a homomorphism defined in~\cite[Proposition 3.7]{Gon}:

\begin{proposition}\label{morphism.miguel} Let $\varphi$ be as in Notation~\ref{setup}. Then
there exists a homomorphism
\begin{equation*}
 H^0(\mathcal N_\varphi) \overset{\Psi}\longrightarrow \mathrm{Hom}(\pi^*(\mathcal I/\mathcal I^2), \mathcal O_X),
\end{equation*}
that appears when taking cohomology on the commutative diagram~\cite[(3.3.2)]{Gon}. Since
\begin{equation*}
\mathrm{Hom}(\pi^*(\mathcal I/\mathcal I^2), \mathcal O_X)=\mathrm{Hom}(\mathcal I/\mathcal I^2, \pi_*\mathcal O_X)=\mathrm{Hom}(\mathcal I/\mathcal I^2, \mathcal O_Y) \oplus \mathrm{Hom}(\mathcal I/\mathcal I^2, \mathcal E),
\end{equation*}
the homomorphism $\Psi$ has two components
\begin{eqnarray*}
H^0(\mathcal N_\varphi) & \overset{\Psi_1}  \longrightarrow & \mathrm{Hom}(\mathcal I/\mathcal I^2, \mathcal O_Y) \cr
H^0(\mathcal N_\varphi) & \overset{\Psi_2}\longrightarrow & \mathrm{Hom}(\mathcal I/\mathcal I^2, \mathcal E).
\end{eqnarray*}
\end{proposition}

\noindent We also need to recall the definition of rope:

\begin{definition}\label{defrope}
Let $\Upsilon$ be a reduced connected scheme and let $\mathcal F$ be a locally free sheaf of rank $m-1$ on $Y$.
A rope of multiplicity $m$ or an $m$-rope, for short, on $\Upsilon$ with conormal bundle $\mathcal F$ is a scheme $\tilde \Upsilon$ with 
\begin{enumerate}
\item
$\SI_{\Upsilon, \tilde \Upsilon}^2=0$ and
\item
$\SI_{\Upsilon, \tilde \Upsilon} \simeq \mathcal F$ as $\SO_{Y}$--modules.
\end{enumerate}
\end{definition}

\begin{proposition}\label{intrinsic}
Let $T$ be a smooth irreducible algebraic curve and let $0$ be a closed point of $T$.
Let $\Phi: \mathcal X \longrightarrow \mathbf P^N_T$ be a flat family of morphisms over $T$ (i.e., $\Phi$ is a $T$--morphism for which $\mathcal X \longrightarrow T$ is proper, flat and surjective) such that
\begin{enumerate}
\item $\mathcal X$ is irreducible and reduced;
\item $\mathcal X_t$ is smooth, irreducible and projective for all $t \in T$;
\item $\mathcal X_0=X$ and $\Phi_0=\varphi$.
\end{enumerate}
Let $\Delta$ be the first infinitesimal neighborhood of $0$ in $T$. Let $\widetilde X=\mathcal X_\Delta$, $\widetilde \varphi = \Phi_\Delta$ and let $\nu$ be the element of $H^0(\mathcal N_\varphi)$ corresponding to $\widetilde \varphi$.
If $\, \Psi_2(\nu) \in \ \mathrm{Hom}(\mathcal I/\mathcal I^2, \mathcal E)$ is a surjective homomorphism, then, shrinking $T$ if necessary, $\Phi_t$ is a closed immersion for any $t \in T, t \neq 0$.
\end{proposition}

\begin{proof}
Recall (see~\cite[Proposition 2.1.(1)]{Gon}) that $\Psi_2(\nu)$ corresponds to a pair $(\widetilde Y, \tilde i)$ where $\widetilde Y$ is a rope on $Y$ with conormal bundle $\mathcal E$ (therefore, in particular, $\widetilde Y$ is a rope of multiplicity $n$ on $Y$) and $\tilde i: \widetilde Y \longrightarrow \mathbf P^N$ extends $i$. Moreover, since $\Psi_2(\nu)$ is a surjective homomorphism, $\tilde i$ is an embedding (see~\cite[Proposition 2.1.(2)]{Gon}). Now ~\cite[Theorem 3.8.(1)]{Gon} implies that $(\mathrm{im}\,\widetilde \varphi)_0=\tilde i(\widetilde Y)=\widetilde Y$.
Let $L=\varphi^*\mathcal O_{\mathbf P^N}(1)$ and $\mathcal L=\Phi^*\mathcal O_{\mathbf P^N_T}(1)$. Since $\mathcal E$ is both the conormal bundle of $\widetilde Y$ and the trace--zero module of $\pi$,
\begin{equation}\label{Xi1}
\chi(X,L^{\otimes l})=\chi(\widetilde Y,\mathcal O_{\widetilde Y}(l)),
\end{equation}
for all $l \geq 0$.
Let $\mathcal Y=\Phi(\mathcal X)$ and define $\Pi: \mathcal X \longrightarrow \mathcal Y$ so that $\Phi$ factors through $\Pi$.
Since $(\mathrm{im}\,\widetilde \varphi)_0=\widetilde Y$ and $(\mathrm{im}\,\widetilde \varphi)_0 \subset \mathcal Y_0$, we have an exact sequence
\begin{equation}\label{proj}
0 \longrightarrow \mathcal K \longrightarrow \mathcal O_{\mathcal Y_0} \longrightarrow \mathcal O_{\widetilde Y} \longrightarrow 0,
\end{equation}
which implies that
\begin{equation}\label{Xi2}
\chi(\mathcal Y_0,\mathcal O_{\mathcal Y_0}(l)) \geq \chi(\widetilde Y,\mathcal O_{\widetilde Y}(l)),
\end{equation}
for $l \gg 0$.
On the other hand, $\mathcal Y$ is a family, flat over $T$, because $\mathcal X$ is irreducible and reduced and so is $\mathcal Y=\Phi(\mathcal X)$. 
Thus the Hilbert polynomials of the fibers of $\mathcal Y$ are all equal,
i.e.,
\begin{equation}\label{Xi3}
\chi(\mathcal Y_0,\mathcal O_{\mathcal Y_0}(l)) = \chi(\mathcal Y_t,\mathcal O_{\mathcal Y_t}(l))
\end{equation}
for $l \geq 0$.
Now, since $\varphi$ is finite, after shrinking $T$ if necessary, so is $\Phi$.
Moreover, since $\mathcal Y$ is irreducible and reduced, its general fibers over $T$ are also irreducible and reduced.
Thus, for general $t \in T$, the induced map $\mathcal X_t \longrightarrow \mathcal Y_t$ is a finite morphism between integral varieties of the same dimension so it is surjective, i.e. $\Phi_t(\mathcal X_t)=\mathcal Y_t$.
Therefore, there is an inclusion
\begin{equation}\label{incl}
0 \longrightarrow \mathcal O_{\mathcal Y_t} \longrightarrow {\Pi_t}_*\mathcal O_{\mathcal X_t} \longrightarrow \mathcal Q \longrightarrow 0,
\end{equation}
and, by the projection formula we have
\begin{equation}\label{Xi4}
\chi(\mathcal Y_t,\mathcal O_{\mathcal Y_t}(l)) \leq  \chi(\mathcal X_t,\mathcal L_t^{\otimes l})
\end{equation}
for $l \gg 0$.
Finally, since $\mathcal X$ is flat over $T$  and $\mathcal L$ is ample we also have
\begin{equation}\label{Xi5}
\chi(X,L^{\otimes l}) = \chi(\mathcal X_t,\mathcal L_t^{\otimes  l})
\end{equation}
for infinitely many positive values of $l$.

\noindent
Putting~\eqref{Xi1},~\eqref{Xi2},~\eqref{Xi3},~\eqref{Xi4} and~\eqref{Xi5} together we conclude that~\eqref{Xi2} and~\eqref{Xi4} are equalities for infinitely many positive values of $l$.
Now, looking at~\eqref{proj}, since~\eqref{Xi2} is an equality we see that $h^0(\mathcal K(l))=\chi(\mathcal K(l))=0$
for infinitely many positive values of $l$. Then $\mathcal K$, which is the ideal sheaf of $\widetilde Y$ inside
$\mathcal Y_0$, is $0$ so
\begin{equation}\label{central.fiber}
\widetilde Y=\mathcal Y_0.
\end{equation}
Similarly, looking at~\eqref{incl}, since~\eqref{Xi4} is an equality we see that $h^0(\mathcal Q(l))=\chi(\mathcal Q(l))=0$ for infinitely many positive values of $l$. Then $\mathcal Q=0$, so
\begin{equation}\label{isom}
 {\Pi_t}_*\mathcal O_{\mathcal X_t} \simeq \mathcal O_{\mathcal Y_t}.
 \end{equation}
In particular, ${\Pi_t}_*\mathcal O_{\mathcal X_t}$ is a locally free sheaf on $\mathcal Y_t$. On the other hand, $\Pi_t$ is finite, since $\pi$ is, so by~\cite[3.G]{Mat} (see also~\cite[Corollary 5.5]{KM}), $\Pi_t$ is flat. Finally,~\eqref{isom}  implies (see~\cite[Corollary III.11.3]{Hart}) that the fibers of $\Pi_t$ are connected, so $\Pi_t$ is an isomorphism and $\Phi_t$ is an embedding.
\end{proof}

\smallskip
\noindent
We use Proposition~\ref{intrinsic} to obtain the next theorem:

\begin{theorem}\label{cor.morphisms}
Let $\widetilde \varphi: \widetilde X \longrightarrow \mathbf P^N_\Delta$ be a first order infinitesimal deformation of $\varphi$ ($\Delta$ is Spec$(\kk[\epsilon]/\epsilon^2)$) and let $\nu$ be the class of $\widetilde \varphi$ in $H^0(\mathcal N_\varphi)$. If
\begin{enumerate}
\item there exists an algebraic formally semiuniversal deformation of $\varphi$ and $\varphi$ is unobstructed; and
\item $\Psi_2(\nu) \in \ \mathrm{Hom}(\mathcal I/\mathcal I^2,\mathcal E)$ is a surjective homomorphism,
\end{enumerate}
then there exist a flat family of morphisms, $\Phi: \mathcal X \longrightarrow \mathbf P^N_T$ over $T$, where $T$ is a smooth irreducible algebraic curve, and a closed point $0 \in T$, such that
\begin{enumerate}
\item[(a)] $\mathcal X_t$ is a smooth, irreducible, projective variety for all $t \in T$;
\item[(b)] the restriction of $\Phi$ to the first infinitesimal neighborhood of $0$ is $\widetilde \varphi$ (and hence $\Phi_0=\varphi$); and
\item[(c)] for any $t \in T$, $t \neq 0$, $\Phi_t$ is a closed immersion into $\mathbf P^N_t$.
\end{enumerate}
\end{theorem}

\begin{proof}
There exists a smooth algebraic variety $M$ which is the base of the algebraic formally semiuniversal deformation
of $\varphi$ and $H^0(\mathcal N_\varphi)$ is the tangent space of $M$ at $[\varphi]$.
Then $\nu$ represents a tangent vector to $M$ at $[\varphi]$. Thus there is a smooth irreducible algebraic curve $T \subset M$ passing through $[\varphi]$ and a family of morphisms over $T$,  $\Phi: \mathcal X \longrightarrow \mathbf P^N_T$ satisfying (1), (2) and (3) in Proposition~\ref{intrinsic} and $\Phi_\Delta=\widetilde \varphi$. Now, since $\Psi_2(\nu)$ is a surjective homomorphism, Proposition~\ref{intrinsic} implies (c).
\end{proof}

\noindent Theorem~\ref{cor.morphisms} has this straight--forward corollary:

\begin{corollary}\label{cor2.morphisms}
 If $\varphi$ satisfies Condition (1) of Theorem~\ref{cor.morphisms} and the image of $\Psi_2$ in $\mathrm{Hom}(\mathcal I/\mathcal I^2,\mathcal E)$ contains a surjective homomorphism, then $\varphi$ can be deformed to an embedding.
\end{corollary}

\begin{remark}\label{cor.morphisms.2}
{\rm Theorem~\ref{cor.morphisms} holds in a very general setting. Indeed, $\varphi$ possesses an 
algebraic formally semiuniversal deformation  if $X$ itself possesses an algebraic formally semiuniversal deformation. 
This is because $\varphi$ is non--degenerate by generic smoothness 
(for the definition of non--degenerate morphism, see~\cite[p. 376]{Hor1} or~\cite[Definition 3.4.5]{Ser}). 
On the other hand, the existence of an algebraic formally semiuniversal deformation of $X$ follows if $X$ has 
an effective formal semiuniversal deformation, because the latter is algebraizable by 
Artin's algebraization theorem (see \cite{Artin.alg}). For example, if $\varphi$ is the canonical map 
of a smooth variety of general type $X$ with ample and base--point-free canonical bundle, $\varphi$ possesses 
an algebraic formally semiuniversal deformation (see \cite[Lemma 2.3]{canonical}).}
\end{remark}

\noindent If $X$ is a curve, then the existence of algebraic formally semiuniversal deformations is trivial and, in order to know if $\varphi$ is unobstructed, it suffices to check some cohomology vanishings on $Y$. Thus we end this section obtaining from Theorem~\ref{cor.morphisms} this neat statement regarding morphisms from curves:

\begin{theorem}\label{curves}
Assume $Y$ is a curve. Let $\widetilde \varphi: \widetilde X \longrightarrow \mathbf P^N_\Delta$ be a first order infinitesimal deformation of $\varphi$ ($\Delta$ is Spec$(\kk[\epsilon]/\epsilon^2)$) and let $\nu$ be the class of $\widetilde \varphi$ in $H^0(\mathcal N_\varphi)$.
If
\begin{enumerate}
\item $h^1(\SO_Y(1))=0, % \;\; \text{and} \;\;
\, h^1(\SE \otimes \SO_Y(1))=0$; and
\item $\Psi_2(\nu) \in \ \mathrm{Hom}(\mathcal I/\mathcal I^2,\mathcal E)$ is a surjective homomorphism,
\end{enumerate}
then there exist a flat family of morphisms, $\Phi: \mathcal X \longrightarrow \mathbf P^N_T$ over $T$, where $T$ is a smooth irreducible algebraic curve, and a closed point $0 \in T$, such that
\begin{enumerate}
\item[(a)] $\mathcal X_t$ is a smooth, irreducible, projective curve for all $t \in T$;
\item[(b)] the restriction of $\Phi$ to the first infinitesimal neighborhood of $0$ is $\widetilde \varphi$ (and hence $\Phi_0=\varphi$); and
\item[(c)] for any $t \in T$, $t \neq 0$, $\Phi_t$ is a closed immersion into $\mathbf P^N_t$.
\end{enumerate}
\end{theorem}

\begin{proof}
Condition (2) of Theorem~\ref{cor.morphisms} holds by assumption. We check now that Condition (1) of Theorem~\ref{cor.morphisms} also holds. Since $X$ is a curve, then $h^2(\SO_X)=0$ and $X$ has an algebraic formally semiuniversal deformation. By Remark~\ref{cor.morphisms.2}, $\varphi$ also has an algebraic formally semiuniversal deformation.
On the other hand, (1) implies $h^1(\pi^{\ast}\SO_Y(1))=0$, so from the pullback to $X$ of the Euler sequence we get $h^1(\varphi^{\ast}\ST_{\PP^N})=0$. In addition $h^2(\ST_X)=0$, so it follows from the sequence~\eqref{normal.sheaf} that $h^1(\mathcal N_\varphi)=0$. Thus $\varphi$ is unobstructed so Condition (1) of Theorem~\ref{cor.morphisms} is satisfied.
\end{proof}

\medskip
\noindent
\begin{remark}\label{compare.Compositio}
{\rm We compare Theorem~\ref{curves} with~\cite[Theorem 1.1]{GGP}. In a sense, \cite[Theorem 1.1]{GGP} can be considered stronger because no condition is required from $\Psi_2(\nu)$. However the statement of~\cite[Theorem 1.1]{GGP} needs the hypothesis
\begin{equation}\label{ged.N+1}
 h^0(\pi^{\ast}\SO_Y(1))\geq N+1,
\end{equation}
 which is not required in Theorem~\ref{curves} (under the hypothesis of Theorem~\ref{curves}, \eqref{ged.N+1} holds nevertheless if, for instance, the embedded rope on $Y$ that corresponds to $\Psi_2(\nu)$ is not contained in a hyperplane of $\mathbf P^N$). If one compares the proofs of Theorem~\ref{curves} and~\cite[Theorem 1.1]{GGP}, one sees that the latter is more ad--hoc and involved as it uses properties of the Hilbert scheme and the moduli of curves. In contrast, the proof of Theorem~\ref{curves} comes from more general principles and fits in the general setting dealt by Theorem~\ref{cor.morphisms}, which applies to varieties of arbitrary dimension.}
\end{remark}

\begin{remark}
{\rm In view of Remark~\ref{compare.Compositio} and since Hom$(\mathcal I/\mathcal I^2,\mathcal E)$ possesses elements which are not surjective homomorphisms, \cite[Theorem 1.1]{GGP} shows that the surjectivity of $\Psi_2(\nu)$  (see Proposition~\ref{intrinsic} and Theorem~\ref{cor.morphisms}, (2)) is not in general a necessary condition for $\varphi$ to be deformed  to an embedding (for a geometric application of this to the smoothing of certain non locally Cohen--Macaulay multiple structures, see~\cite{nonCM}).}
\end{remark}

\section{A first application: smoothing ropes}\label{ropes.section}

\medskip

In this section we use the results of Section~\ref{intrinsic.section} to obtain results on the smoothing of ropes (recall Definition~\ref{defrope}). More precisely, we apply Proposition~\ref{intrinsic} to obtain a sufficient condition, Theorem~\ref{cor.ropes} to smooth ropes. As a consequence, we obtain Proposition~\ref{canonical.carpets.smoothing} on the smoothing of canonically embedded carpets. We also give simpler, more conceptual and unified proofs of~\cite[Theorem 2.4]{GGP}, \cite[Theorem 3.5]{GP} and \cite[Theorem 3.5]{GGP2} on the smoothing of ropes on curves and $K3$--carpets. We start by making precise what we mean by smoothing a rope:

\begin{definition}
With the notation of~\ref{setup}, let $\WY$ be a rope on a smooth variety $Y$ and let $\tilde i: \widetilde Y \hookrightarrow \mathbf P^N$ be an embedding of $\widetilde Y$ in projective space that extends $i$. We say that $\tilde i(\WY)$ can be smoothed inside $\mathbf P^N$ if there exist a smooth irreducible algebraic curve $T$, a closed point $0 \in T$ and a closed subscheme $\mathcal Y$ of $\mathbf P^N_T$, flat over $T$, such that $\tilde i(\widetilde Y)=\mathcal Y_0$ and $\mathcal Y_t$ is a  smooth and irreducible projective variety  if $t \neq 0$.
\end{definition}

\begin{theorem}\label{cor.ropes}
With the notation of~\ref{setup}, let $\widetilde Y$ be a (multiplicity $n$) rope on $Y$ with conormal bundle $\mathcal E$. Let $\tilde i: \widetilde Y \hookrightarrow \mathbf P^N$ be an embedding of $\widetilde Y$ in projective space that extends $i$ and let $\nu_2$ be the element of $\mathrm{Hom}(\mathcal I/\mathcal I^2,\mathcal E)$ corresponding to $(\widetilde Y,\tilde i)$ (see~\cite[Proposition 2.1 (1)]{Gon}).
If
\begin{enumerate}
\item there exists an algebraic formally semiuniversal deformation of $\varphi$ and $\varphi$ is unobstructed; and
\item $\nu_2$ is in the image of $\Psi_2$,
\end{enumerate}
then $\tilde i(\WY)$ can be smoothed inside $\mathbf P^N$.
\end{theorem}

\begin{proof}
Let $\nu$ be an element of $H^0(\mathcal N_\varphi)$ such that $\Psi_2(\nu)=\nu_2$ and let $\widetilde \varphi$
be the first order infinitesimal deformation of $\varphi$ associated to $\nu$ (see~\cite[(3.3.2)]{Gon}). As in the proof of Theorem \ref{cor.morphisms}, there is a smooth irreducible algebraic curve $T \subset M$ passing through $[\varphi]$ and a family of morphisms over $T$, $\Phi: \mathcal X \longrightarrow \mathbf P^N_T$ satisfying (1), (2) and (3) in Proposition~\ref{intrinsic} and $\Phi_\Delta=\widetilde \varphi$. Now let $\mathcal Y=\Phi(\mathcal X)$. Since $\tilde i$ is an embedding, $\nu_2$ is surjective (see~\cite[Proposition 2.1 (2)]{Gon} or~\cite{HV}), so $\Psi_2(\nu)=\nu_2$ implies, by Proposition~\ref{intrinsic}, that $\Phi_t$ is a closed immersion for any $t \in T$ general. Thus $\mathcal Y_t=\Phi_t(\mathcal X_t)$ is a smooth and irreducible variety. Finally,~\eqref{central.fiber} says that $\mathcal Y_0=\tilde i(\widetilde Y)$.
\end{proof}

\medskip

\noindent As a first consequence of Theorem~\ref{cor.ropes} we will prove a result,  Proposition~\ref{canonical.carpets.smoothing},
on the smoothing of canonically embedded carpets. Before that, we need to define these carpets and to characterize them. Examples of canonical carpets that can be smoothed are the complete intersection of a smooth quadric a doubled quadric in $\mathbf P^4$ and, by Proposition~\ref{canonical.carpets.smoothing}, the canonical carpets of Corollary~\ref{canonical.carpet}.

\begin{definition}\label{defcanonicalcarpet}
Let $Y$ be a smooth surface and let $i: Y  \hookrightarrow \mathbf P^N$ be an embedding of $Y$ in $\mathbf P^N$. Let $\widetilde Y$ be a multiplicity $2$ rope supported on $Y$ and let ${\tilde i}: \widetilde Y  \hookrightarrow \mathbf P^N$ be an embedding of $\WY$ in $\mathbf P^N$ that extends $i$. We say that $\WY$ is canonically embedded by $\tilde i$ or, for short, that $\tilde i(\WY)$ is a canonical carpet, if the dualizing sheaf $\omega_{\widetilde Y}$ of $\WY$ is very ample and $\tilde i$ is induced by the complete linear series of $\omega_{\widetilde Y}$.
\end{definition}

\begin{proposition}\label{canonical.carpet.char}
With the same notation of Definition~\ref{defcanonicalcarpet},
if $\tilde i(\widetilde Y)$ is a canonical carpet then its conormal bundle is $\omega_Y(-1)$. If we assume in addition that
\begin{enumerate}
\item[(a)] $Y$ is a regular surface with $p_g(Y)=0$;
\item[(b)] $i$ is induced by a complete linear series on $Y$; and
\item[(c)] $h^1(\mathcal O_Y(1))=0$,
\end{enumerate}
then the converse is also true.
\end{proposition}

\begin{proof}
The first claim follows from ~\cite[(1.4.2)]{GGP2}. The second claim follows from (a), (b), (c) and~\cite[Remark 1.3 and Lemma 1.4]{GGP2}, arguing similarly to the proof of~\cite[Proposition 1.5]{GGP2}.
\end{proof}

\begin{proposition}\label{canonical.carpets.smoothing}
 Let $Y$ be a smooth surface embedded in $\mathbf P^N$ by a morphism $i$ and let $\widetilde Y$ be a canonical carpet supported on $i(Y)$. If
\begin{enumerate}
\item[(a)] $p_g(Y)=0$;
\item[(b)] $i$ is induced by a complete linear series on $Y$; and
\item[(c)] $|\omega_Y^{-2}(2)|$ has a smooth divisor $B$,
 \end{enumerate}
then there exists a double cover $\pi: X \to Y$ branched along $B$ such that $X$ is an irreducible, smooth surface of general type with ample and base--point--free canonical divisor and whose canonical map $\varphi$ factors as $i \circ \pi$.

\smallskip

\noindent If, in addition,
\begin{enumerate}
\item[(d)] $\varphi$ is unobstructed; and
\item[(e)] $h^1(\mathcal O_B(B))=0$, 
\end{enumerate}
then
$\WY$ can be smoothed inside $\mathbf P^N$. More precisely, there exist a smooth algebraic curve $T$, a closed point $0 \in T$ and a family of subschemes $\SY \subset \PP_T^N$, flat over $T$, such that
\begin{enumerate}
\item
if $t \neq 0$, then $\SY_t$ is a smooth, irreducible, canonically embedded surface in $\PP^N$, and
\item if $t=0$, then $\SY_0$ is $\WY$.
\end{enumerate}
\end{proposition}

\begin{proof}
Since $B$ is smooth, we can construct a smooth double cover $\pi: X \longrightarrow Y$ whose trace zero module $\mathcal E$ is $\omega_Y(-1)$, which is, by Proposition~\ref{canonical.carpet.char}, the conormal bundle of $\WY$. The canonical bundle of $X$ is $\pi^*\mathcal O_Y(1)$ so $X$ is a smooth surface of general type with ample and base--point--free canonical divisor. By (a) and (b) the canonical morphism $\varphi$ of $X$ factors as $\varphi= i \circ \pi$.
Since $X$ is a surface of general type with ample and base--point--free canonical divisor, Condition (1) of Theorem~\ref{cor.ropes} holds once we have in account Remark~\ref{cor.morphisms.2} and (d). On the other hand, by~\cite[(2.11)]{canonical} and (e), $h^1(\mathcal N_\pi)=0$, so $\Psi_2$ is surjective (see~\cite[(3.3.2)]{Gon}) and Condition (2) of Theorem~\ref{cor.ropes} also holds.
\end{proof}

\medskip

\noindent
When we specialize Theorem~\ref{cor.ropes} to curves, hypotheses become simplified and we get a nice, clean statement. This result is in fact~\cite[Theorem 2.4]{GGP} but the proof given here is much clearer and more conceptual:

\begin{theorem}\label{curves.2}
Let $Y$ be a curve. Let $\widetilde Y$ be a (multiplicity $n$) rope on $Y$ with conormal bundle $\mathcal E$. Assume there is an embedding $\, \tilde i: \widetilde Y \hookrightarrow \mathbf P^N \,$ of $\, \widetilde Y \,$ in projective space that extends $i$.
If
\begin{equation}\label{curves.2.equation}
h^1(\SO_Y(1))=0 \;\; \text{and}\;\; h^1(\SE \otimes \SO_Y(1))=0,
\end{equation}
then $\tilde i(\WY)$ can be smoothed inside $\mathbf P^N$.
\end{theorem}

\begin{proof}
Since $Y$ is a curve and because of~\eqref{curves.2.equation}, arguing as in the proof of Theorem~\ref{curves} we see that Condition (1) of the statement of Theorem~\ref{cor.ropes} is satisfied. Also, since $Y$ is a curve,  $h^1(\mathcal N_\pi)=0$, so $\Psi_2$ is surjective (see~\cite[(3.3.2)]{Gon}) and Condition (2) of Theorem~\ref{cor.ropes} also holds.
\end{proof}

\noindent We end this section revisiting the smoothing results of $K3$ carpets on rational normal scrolls and on Enriques surfaces, \cite[Theorem 3.5]{GP} and~\cite[Theorem 3.5]{GGP2}. We see how these theorems fit nicely in the general framework brought by the results of Section~\ref{intrinsic.section}. To do so, we reprove a key point in the proof of~\cite[Theorem 3.5]{GP} using Proposition~\ref{intrinsic} and we see how the proof of~\cite[Theorem 3.5]{GGP2} fits also in the setting of Proposition~\ref{intrinsic}. Before all this we recall the definition of $K3$ carpet:

\begin{definition}\label{defK3carpet} Let $Y$ be a smooth surface. We will say that a multiplicity $2$ rope $\widetilde Y$ on $Y$ is a $K3$ carpet if
\begin{enumerate}
\item  the dualizing sheaf $\omega_{\widetilde Y}$ is trivial and
\item $h^1(\mathcal O_{\widetilde Y})=0$.
\end{enumerate}
\end{definition}

\begin{theorem}\label{K3carpetsF0F1F2}
Let $Y$ be a Hirzebruch surface $\mathbf F_e$, for $e=0,1$ or $2$. Let $i: Y \hookrightarrow \PP^g$ be an embedding of $Y$ as a rational normal scroll in $\PP^g$, $g \geq 3$, and let $\WY \subset \PP^g$ be (the only; see~\cite[Theorem 1.3]{GP}) $K3$ carpet on $i(Y)$. Then there exist a smooth analytic curve $T$, a point $0 \in T$ and a family of subschemes $\SY \subset \PP_T^g$, flat over $T$, such that
\begin{enumerate}
\item
if $t \neq 0$, then $\SY_t$ is a smooth irreducible projective $K3$ surface in $\PP^g$, and
\item
if $t=0$, then $\SY_0$ is $\WY$.
\end{enumerate}
\end{theorem}

\begin{proof}
{\it Step 1.} We want to use Proposition~\ref{intrinsic}. For that we need to construct a morphism $\varphi$ and we need 
to produce a suitable first order infinitesimal deformation of $\varphi$. Recall (see~\cite[Proposition 1.5]{GGP2}) that 
the conormal bundle of $\WY$ is $\omega_Y$. Since $Y$ is a Hirzebruch surface $\mathbf F_e$, 
$\omega_Y=\mathcal O_Y(-2C_0-(e+2)f)$, where $C_0$ is a minimal section of $\mathbf F_e$ and $f$ is a fiber. 
Since $e=0,1$ or $2$, $\omega_Y^{-2}$ is base--point--free, so we can choose a smooth divisor $B \in |\omega_Y^{-2}|$ 
and there exists a smooth $K3$ double cover $\pi: X \longrightarrow Y$ branched along $B$ and with trace zero module $\omega_Y$. Let $L=\pi^*\mathcal O_Y(1)$ and let $\varphi$ be the morphism induced on $X$ by $|L|$. Then $\varphi = i \circ \pi$ and $\mathcal E=\omega_Y$, in the notation of~\ref{setup}. Let $\tilde i$ be the embedding of $\WY$ in $\PP^g$ and let $\mu$ be the element in Hom$(\mathcal I/\mathcal I^2, \mathcal E)$ that corresponds to $(\WY,\tilde i)$.
Recall that $H^1(\mathcal N_\pi)=H^1(\mathcal O_B(B))$ (see~\cite[(2.11)]{canonical}). Since $e=0,1$ or $2$, $H^1(\omega_Y^{-2})=0$, so $H^1(\mathcal N_\pi)=0$. Then~\cite[Lemma 3.3]{Gon} implies that $\Psi_2$ surjects onto Hom$(\mathcal I/\mathcal I^2, \mathcal E)$. Thus there exists $\nu \in H^0(\mathcal N_\varphi)$ such that $\Psi_2(\nu)=\mu$.
We consider $\widetilde \varphi: \widetilde X \to \mathbf P^g_\Delta$ to be the first order infinitesimal deformation of $\varphi$ that corresponds to $\nu$.

\smallskip

\noindent {\it Step 2.} We want to apply Proposition~\ref{intrinsic} to $\varphi$. For that we need to construct a suitable family $\Phi$ of morphisms. We fix a marking for $(X,L)$ and consider the moduli $\mathcal M_g$ of  marked polarized  $K3$ surfaces of genus $g$ (for details on this moduli space see e.g.~\cite{SP}). Let $\widetilde L=\widetilde \varphi^*\mathcal O_{\mathbf P^N_\Delta}(1)$. Then $(\widetilde X, \widetilde L)$ corresponds to a tangent vector $v$ to $\mathcal M_g$ at $[(X,L)]$. Taking a path $T$ through $[(X,L)]$ and tangent to $v$ we can construct a family $(\mathcal X,\mathcal L)$ extending $(\widetilde X, \widetilde L)$, so that $\mathcal L$ induces a morphism $\Phi: \mathcal X \to \mathbf P^g_T$ that extends $\widetilde \varphi$ and satisfies (1), (2) and (3) of Proposition~\ref{intrinsic}. Note that, since the moduli of marked polarized $K3$ surface is analytic, the path $T$ above is a smooth analytic curve rather than an algebraic curve, but the arguments of the proof of Proposition~\ref{intrinsic} work the same if $T$ is an analytic curve. Since $\Psi_2(\nu)=\mu$ is a surjective homomorphism (see~\cite[Proposition 2.1 (2)]{Gon}) of Hom$(\mathcal I/\mathcal I^2, \mathcal E)$, Proposition~\ref{intrinsic} implies that $\Phi_t$ is an embedding if $t \neq 0$. Let $\mathcal Y=\Phi(\mathcal X)$. Then~\eqref{central.fiber} says that $\mathcal Y_0=\widetilde Y$, so $\mathcal Y$ is the family we were looking for.
\end{proof}

\begin{corollary}(\cite[Theorem 3.5]{GP})
 Let $Y$ be a rational normal scroll in $\PP^g$, $g \geq 3$ and let $\WY \subset \PP^g$ be (the only; see~\cite[Theorem 1.3]{GP}) $K3$ carpet on $Y$. Then there exist a smooth analytic curve $T$, a point $0 \in T$ and a family of subschemes $\SY \subset \PP_T^g$, flat over $T$, such that
\begin{enumerate}
\item
if $t \neq 0$, then $\SY_t$ is a smooth irreducible projective $K3$ surface in $\PP^g$, and
\item
if $t=0$, then $\SY_0$ is $\WY$.
\end{enumerate}
\end{corollary}

\begin{proof}
If $Y=\mathbf F_0, \mathbf F_1$ or $\mathbf F_2$, the corollary is Theorem~\ref{K3carpetsF0F1F2}. If $Y=\mathbf F_e$ with $e \geq 3$, then the corollary follows from Theorem~\ref{K3carpetsF0F1F2} and~\cite[Theorem 3.6]{GP}.
\end{proof}

\begin{theorem}
Let $Y$ be an Enriques surface embedded in $\PP^N$ by a morphism $i$ and let $\WY \subset \PP^N$ be a projective $K3$ carpet on $i(Y)$. Then $\tilde i(\WY)$ can be smoothed inside $\mathbf P^N$. More precisely, there exist a smooth algebraic curve $T$, a point $0 \in T$ and a family of subschemes $\SY \subset \PP_T^N$, flat over $T$, such that
\begin{enumerate}
\item if $t \neq 0$, then $\SY_t$ is a smooth irreducible projective $K3$ surface in $\PP^N$, and
\item if $t=0$, then $\SY_0$ is $\WY$.
\end{enumerate}
\end{theorem}

\begin{proof} We construct $\varphi$ and $\widetilde \varphi$ as in Step 1 of the proof of Theorem~\ref{K3carpetsF0F1F2}. The only differences are that now $\varphi$ is the morphism induced by the pullback of $H^0(\mathcal O_Y(1))$ by $\pi$ and that, since $Y$ is an Enriques surface, $\pi$ is \'etale, so $H^1(\mathcal N_\pi)$ also vanishes as in Theorem~\ref{K3carpetsF0F1F2}. Then after performing the same arguments of Theorem~\ref{K3carpetsF0F1F2}, Step 1, we have that $\Psi_2(\nu)=\mu$, where $\nu \in H^0(\mathcal N_\varphi)$ corresponds to $\widetilde \varphi$, $\mu \in $ Hom$(\mathcal I/\mathcal I^2,\mathcal E)$ corresponds to $\WY \subset \PP^N$ and $\mu$ is surjective. Now we would like to construct a suitable flat family $\Phi: \mathcal X \longrightarrow \mathbf P^N_T$ of morphisms over an algebraic curve $T$, extending $\widetilde \varphi$ and satisfying (1), (2) and (3) of Proposition~\ref{intrinsic}. For this we use the same Hilbert scheme arguments of the proof of~\cite[Theorem 3.2]{GGP2}. Since $\Psi_2(\nu)=\mu$ is a surjective homomorphism, Proposition~\ref{intrinsic} implies that $\Phi_t$ is an embedding if $t \neq 0$ (note that this fact was shown in the proof of~\cite[Theorem 3.2]{GGP2} using an ad--hoc argument). We set $\mathcal Y=\Phi(\mathcal X)$. Then~\eqref{central.fiber} says that $\mathcal Y_0=\widetilde Y$, so $\mathcal Y$ is the family we were looking for.
\end{proof}

\section{A second application: constructing surfaces of general type}\label{canonical.surfaces.section}

In this section we illustrate how Theorem~\ref{cor.morphisms} can be used to produce smooth varieties of given invariants embedded in projective space. Precisely, we construct canonically embedded surfaces with $c_1^2=3p_g-7$ by deforming canonical double covers of certain non minimal rational surfaces. This construction was previously done by Ashikaga and Konno (see~\cite[4.5]{AK}) by ad--hoc methods. Here we revisit it and reveal it as a particular case of the general construction described by Theorem~\ref{cor.morphisms}. In addition to the notation~\ref{setup}, we will use the following:

\begin{noname}\label{setup.AK}
{\rm {\bf Notation.} We will use this notation in this section.
\begin{enumerate}
 \item Given natural numbers $a \leq b \leq c$  we set $r=a+b+c+3$ and we consider $Z=S(a,b,c)$ to be the smooth rational normal scroll of dimension $3$ and degree $r-3$ in $\mathbf P^{r-1}$ obtained by embedding $\mathbf P= \mathbf P(\mathcal O_{\mathbf P^1}(a) \oplus \mathcal O_{\mathbf P^1}(b) \oplus \mathcal O_{\mathbf P^1}(c))$ by $|\mathcal O_\mathbf P(1)|$.
\item We denote by $H$ the hyperplane divisor of $Z$ and by $F$ the fiber of the projection of $Z$ to $\mathbf P^1$.
\item For any $k$ such that $2a+k \geq 0$ and $a +2k+r-5 >0$, we consider $L=\mathcal O_Z(2H+kF)$.
\end{enumerate}
In addition to the conventions in~\ref{setup}, $X, Y, \pi$ and $i$ will satisfy the following:

\begin{enumerate}

\item[(4)] $Y$ will be a general member of $|L|$ and $i: Y \to \mathbf P^N$ is induced by
$|\omega_Y \otimes L|$ (note that $\omega_Z \otimes L^{\otimes 2}$ is very ample by~\ref{setup.AK}, (3)); recall that by $\mathcal O_Y(1)$ we  mean $i^*\mathcal O_{\mathbf P^N}(1)$.
\item[(5)] We consider a general member $B \in |\omega_Y^{-2} \otimes \mathcal O_Y(2)|$ and we define $\pi: X \to Y$ as the double cover of $Y$ branched along $B$.
\end{enumerate}}
\end{noname}

\begin{remark}\label{howXandYare}
 {\rm It is easy to see that $X$ and $Y$ are as follows:
\begin{enumerate}
\item The line bundle $L$ on $Z$ is base--point--free and big, because of~\ref{setup.AK}, (3). Therefore $Y$ is smooth and irreducible. In addition, $Y$ is a conic bundle over $\mathbf P^1$, so in particular $Y$ is a non minimal rational surface.
\item Note that $\omega_Y(-1) \simeq L^{-1}|_Y$; since $H^0(L^{-1}|_Y)=0$, $X$ is a connected surface.
\item The line bundle $\omega_Y^{-2} \otimes \mathcal O_Y(2) \simeq L^{\otimes 2}|_Y$ is also base--point--free. Therefore $B$ is smooth  and $X$ is a smooth irreducible surface.
\item Since $p_g(Y)=0$, the canonical map of $X$ is the composition $i \circ \pi$. Therefore $X$ is a surface of general type whose canonical bundle is ample and base--point--free.
\item Since $H^1(L^{-1}|_Y)=0$ and because of~\ref{setup.AK}, (3), $X$ has these invariants:
\begin{enumerate}
\item $p_g(X)=4r+6k-15 \geq 7$,
\item $q(X)=0$, and
\item $c_1^2(X)=3p_g(X)-7$.
\end{enumerate}
\end{enumerate}}
\end{remark}

\begin{proposition}\label{construct.canonical.surfaces} Let $\varphi$ be the canonical map of $X$. Then $\varphi$
\begin{enumerate}
\item is unobstructed, and
\item can be deformed (in a family of morphisms, flat over a smooth algebraic curve) to an embedding into projective space.
\end{enumerate}
\end{proposition}

\begin{proof}
We want to apply Theorem~\ref{cor.morphisms}. We check first Condition (1). Remarks~\ref{cor.morphisms.2} and~\ref{howXandYare}, (4) say that $\varphi$ has an algebraic formally semiuniversal deformation. Now we see that $\varphi$ is unobstructed. \cite[Corollary 2.2, (3)]{canonical} says that $\varphi$ is unobstructed if $X$ is unobstructed.
To prove that $X$ is unobstructed, we consider the  projection $p: X \to \mathbb P^1$ induced by the projection from the ruled variety $Z$ onto  $\mathbb P^1$. Then we look at the forgetful morphism $\mathrm{Def}_p \to \mathrm{Def}_X$ between the functors of infinitesimal deformations of $p$ and of $X$. Indeed, the functor $\mathrm{Def}_X$ is smooth if both $\mathrm{Def}_p$ and the forgetful morphism $\mathrm{Def}_p \to \mathrm{Def}_X$ are. We see first that the forgetful morphism is smooth. For that, we check first 
\begin{equation}\label{tangentP1.pullback.vanish.cohom}
 H^1(p^{\ast} \mathcal T_{\mathbb P^1})=0.
\end{equation}
To prove~\eqref{tangentP1.pullback.vanish.cohom}, we push down $p^{\ast} \mathcal T_{\mathbb P^1}$ to $Y$; then~\eqref{tangentP1.pullback.vanish.cohom} follows from  standard cohomology computations on $Y$ and $Z$, having in account~\ref{setup.AK}, (3).
Finally, by~\cite[Proposition 3.4.11, (iii)]{Ser}, $H^1(p^{\ast} \mathcal T_{\mathbb P^1})=0$ implies that the forgetful morphism is smooth.

\smallskip
\noindent Now we see that $\mathrm{Def}_p$ is also smooth. For this, by~\cite[Lemma 3.4.7, (iv) and Theorem 3.4.8]{Ser}, it suffices to see that $H^2(\ST_{X/\mathbb P^1})=0$. This cohomology group fits in the exact sequence
\begin{equation*}
0 \longrightarrow H^2(\ST_{X/\mathbb P^1}) \longrightarrow H^2(\ST_X) \overset{\mu}\longrightarrow H^2(p^{\ast} \mathcal T_{\mathbb P^1}) \longrightarrow 0,
\end{equation*}
which is injective on the left--hand--side because of~\eqref{tangentP1.pullback.vanish.cohom}. We study now the surjective homomorphism $\mu$. We claim that 
\begin{equation}\label{H2.isom}
H^2(\ST_X) \simeq H^2(p^{\ast} \mathcal T_{\mathbb P^1}).
\end{equation}
In this case $\mu$ would be an isomorphism, so $H^2(\ST_{X/\mathbb P^1})$ would  vanish as wished. 
Thus, to complete the checking of Condition (1) of Theorem~\ref{cor.morphisms} it only remains to prove~\eqref{H2.isom}. 
For this we compute the dimensions of $H^2(\ST_X)$ and $H^2(p^{\ast} \mathcal T_{\mathbb P^1})$. 
For the dimension of $H^2(\ST_X)$ we consider the exact sequence of sheaves on $X$ 
\begin{equation*}
 0 \longrightarrow \mathcal T_X \longrightarrow \pi^*\mathcal T_Y \longrightarrow \mathcal N_\pi \longrightarrow 0. 
\end{equation*}
Clearly $H^2(\mathcal N_\pi)=0$, because $\mathcal N_\pi$ is supported on the ramification of $\pi$, which has dimension $1$.
We claim that
\begin{equation}\label{vanishing.H1.Npi}
 H^1(\mathcal N_\pi)  = 0 
\end{equation}
as well. To see this 
we recall (see~\cite[Lemma 2.5 and (2.11)]{canonical}) that
\begin{equation*}
 H^1(\mathcal N_\pi)  \simeq H^1(\mathcal O_B(B)).
\end{equation*}
Since $p_g(Y)=q(Y)=0$, then
$H^1(\mathcal N_\pi) \simeq H^1(\mathcal O_Y(B))=H^1(L^{\otimes 2}|_Y)$, which vanishes as one can easily 
see using standard cohomological computations on $Z$. Then $H^2(\ST_X)$ is isomorphic to $H^2(\pi^*\mathcal T_Y)$. 
After pushing down $\pi^*\mathcal T_Y$ to $Y$ and carrying out standard cohomological computations on $Y$ and $Z$, one sees that
\begin{equation}\label{dimH2TX}
 h^2(\ST_X)=6k+4r-21.
\end{equation}
On the other hand $H^2(p^{\ast} \mathcal T_{\mathbb P^1})$ can also be shown, again by standard cohomological arguments on $Y$ and $Z$, to be of dimension $6k+4r-21$.
This proves~\eqref{H2.isom} and finishes the checking of Condition (1) of Theorem~\ref{cor.morphisms}.

\medskip
\noindent Now we check Condition (2) of  Theorem \ref{cor.morphisms}. Since $Y \in |L|$ and $\omega_Y(-1) \simeq L^{-1}|_Y$, the normal bundle $\mathcal N_{Y,Z}\otimes \omega_Y(-1)$ is isomorphic to $\mathcal O_Y$, so $\mathcal N_{Y,\mathbf P^N} \otimes \omega_Y(-1)$ has a nowhere vanishing section. This is equivalent to saying that there exists a surjective homomorphism $\nu_2 \in$ Hom$( \SI/\SI^2, \omega_Y(-1))$. By~\eqref{vanishing.H1.Npi} $\nu_2$ lifts to an element $\nu \in H^0(\SN_{\varphi})$ such that  $\Psi_2(\nu)=\nu_2$, so Condition (2) of Theorem~\ref{cor.morphisms} is satisfied and $\varphi$ can be deformed to an embedding as a consequence of Theorem~\ref{cor.morphisms}.
\end{proof}

\noindent We end this section computing the dimension of the irreducible component $\mathcal M'$ of $[X]$ in its moduli space and the codimension in $\mathcal M'$ of the locus parameterizing surfaces whose canonical map is a degree $2$ morphism.

\begin{proposition}\label{moduli}
Let $X$ be as in Notation 3.1. Then the following occurs:
\begin{enumerate}
 \item The point $[X]$ belongs to a unique irreducible component $\mathcal M'$ of its moduli space whose general point  corresponds to a canonically embedded surface.
\item The dimension of the moduli component of $[X]$ is $5p_g(X)+18$.
\item The point $[X]$ belongs to the stratum $\mathcal S$ of $\mathcal M'$ parameterizing surfaces whose canonical map is a degree $2$ morphism; the codimension of  $\mathcal S$ inside $\mathcal M'$ is $1$.
\end{enumerate}
\end{proposition}

\begin{proof}
Since $X$ is unobstructed, $[X]$  only belongs to one irreducible component of the moduli space. By Proposition~\ref{construct.canonical.surfaces}, this component contains a point that corresponds to a canonically embedded surface. Since very ampleness of the canonical bundle is an open condition, the general point of this component does also correspond to a canonically embedded surface. This proves (1).

\smallskip

\noindent Since $X$ is unobstructed, the dimension of the component is $h^1(\mathcal T_X)$. Since $X$ is a surface of general type, $h^0(\mathcal T_X)=0$, so
\begin{equation}\label{dimH1TX}
 h^1(\mathcal T_X)=h^2(\mathcal T_X)-\chi(\mathcal T_X).
\end{equation}
Since $q(X)=0$, Hirzebruch--Riemann--Roch yields
\begin{equation}\label{dimEulerTX}
\chi(\mathcal T_X)=2c_1^2(X)-10p_g(X)-10.
\end{equation}
Then Remark~\ref{howXandYare} (5), \eqref{dimH2TX}, \eqref{dimH1TX} and~\eqref{dimEulerTX} prove (2).

\smallskip

\noindent Finally, to prove (3) observe that $p_g(Y)=q(Y)=h^1(\mathcal O_Y(1))=0$, that $h^0(\omega_Y(-1))=0$ (recall Remark~\ref{howXandYare}, (2)), that $Y$ is unobstructed because $h^1(\mathcal N_{i(Y),\mathbf P^N})=0$ and that $h^1(\omega_Y^{-2}(2))=0$, as seen when proving~\eqref{vanishing.H1.Npi} (recall Remark~\ref{howXandYare}, (3)). Then Conditions (1) to (5) of~\cite[Theorem 2.6]{canonical} are satisfied, so~\cite[Corollary 4.5]{canonical} can be applied in our situation (note that in order to apply~\cite[Corollary 4.5]{canonical} to compute the codimension of $\mathcal S$ inside $\mathcal M'$, the condition $H^1(\mathcal N_\varphi)=0$ required in~\cite[Lemma 4.4]{canonical} is not needed). Then according to~\cite[Corollary 4.5]{canonical}, the codimension of $\mathcal S$ in $\mathcal M'$ is the dimension of Ext$^1(\Omega_Y,\omega_Y(-1))$. This dimension is $1$ as can be easily seen by using standard cohomological arguments on $Y$ and $Z$.
\end{proof}

\medskip

\noindent The arguments of the proof of Proposition~\ref{moduli} yield the following corollary regarding canonical carpets (recall Definition~\ref{defcanonicalcarpet}):

\begin{corollary}\label{canonical.carpet}
\begin{enumerate}
 \item There exists only one canonical carpet $\widetilde Y$ %(recall Definition~\ref{defcanonicalcarpet})
supported on $i(Y)$.
\item The carpet $\widetilde Y$ can be smoothed inside $\mathbf P^N$. More precisely, there exist a smooth algebraic curve $T$, a point $0 \in T$ and a family of subschemes $\SY \subset \PP_T^N$, flat over $T$, such that $\SY_0=\WY$ and $\SY_t$ is a smooth, irreducible, canonically embedded surface of general type in $\PP^N$ if $t \neq 0$.
\end{enumerate}
\end{corollary}

\begin{proof} Notation~\ref{setup.AK}, (4) and Remark~\ref{howXandYare} (1), (2) and (5) imply Conditions (a), (b) and (c) of Proposition~\ref{canonical.carpet.char}. Then, by Proposition~\ref{canonical.carpet.char} and~\cite[Proposition 2.1]{Gon}, in order to prove (1) it suffices to check that Hom$(\mathcal I/\mathcal I^2,\omega_Y(-1))$ has only one surjective homomorphism up to multiplication by scalar. For this recall that the dimension of Ext$^1(\Omega_Y,\omega_Y(-1))$ is $1$, as seen in the last part of the previous proof, and that Hom$(\mathcal I/\mathcal I^2,\omega_Y(-1))$ and Ext$^1(\Omega_Y,\omega_Y(-1))$ are isomorphic (see~\cite[Lemma 3.9]{canonical}). Finally, in the proof of Proposition~\ref{moduli} we saw the existence  of at least one surjective homomorphism in Hom$(\mathcal I/\mathcal I^2,\omega_Y(-1))$.

\smallskip

\noindent Part (2) follows from Remark~\ref{howXandYare}, (3), from the vanishing of $h^1(\mathcal O_B(B))$, which was shown in the proof of Proposition~\ref{moduli}, and from Propositions~\ref{construct.canonical.surfaces}, (1) and~\ref{canonical.carpets.smoothing}.
\end{proof}

\begin{remark}
 {\rm The uniqueness of the canonical carpet $\WY$ on $i(Y)$ allows us to describe it more explicitly. Indeed, if we double $Y$ inside $Z$ and $Z$ is embedded in $\mathbf P^N$ by $|H+(2k+r-5)F|$, by adjunction it is easy to see that we obtain a canonical carpet on $i(Y)$. Then Corollary~\ref{canonical.carpet}, (1) implies that such canonical carpet is $\WY$. Thus, $\WY$ is a member of $|L^{\otimes 2}|$ and a smoothing of $\WY$ can be obtained explicitly deforming $\WY$ inside its complete linear series in $Z$.}
\end{remark}

\begin{remark}{\rm For surfaces $X$ as in~\ref{setup.AK}, Proposition~\ref{moduli}, (3) and Corollary~\ref{canonical.carpet} give these two geometric interpretations for the fact that Hom$(\mathcal I/\mathcal I^2, \mathcal E)$ has dimension $1$: the existence of a unique canonically embedded carpet on $i(Y)$ and the existence of a ``hyperelliptic'' stratum $\mathcal S$, of codimension $1$ in $\mathcal M'$ (we call here $\mathcal S$ hyperelliptic because the canonical map of $X$ is a degree $2$ morphism onto $Y$). This is the same situation that occurs for $K3$ carpets (recall Definition~\ref{defK3carpet}) supported on rational normal scrolls. In that case, any rational normal scroll admits only one $K3$ carpet supported on it. On the other hand, the hyperelliptic locus in the moduli of polarized $K3$ surfaces has, except when the genus is $2$, codimension $1$ (for details, see~\cite{GP}). More in general, the codimension of the ``hyperelliptic'' stratum computed in~\cite[Corollary 4.5]{canonical} can be interpreted, by the same reasons, as the dimension of the space of certain ``canonical'' double structures.}
\end{remark}

\begin{remark}\label{existence}
{\rm For any odd integer $m \geq 7$, there exist surfaces $X$ as in Notation 3.1 with $p_g(X)=m$. Indeed, it is not difficult to find integers $a,b,c,r$ and $k$ satisfying~\ref{setup.AK} (1) and (3) and such that $4r+6k-15=m$. For example,  choose $r=-2k=p_g+15$ and
\begin{enumerate}
\item[] $\frac{p_g+12}{3}=a=b=c$, if $p_g \equiv 0 \ (3)$;
\item[] $\frac{p_g+11}{3}=a=b, c=a+1$, if $p_g \equiv 1 \ (3)$; or
\item[] $\frac{p_g+10}{3}=a, b=c=a+1$, if $p_g \equiv 2 \ (3)$.
\end{enumerate}
Thus Propositions~\ref{construct.canonical.surfaces} and~\ref{moduli} imply, for any odd $p_g \geq 7$ and  $c_1^2=3p_g-7$ , the existence of at least one irreducible component $\mathcal M'$ of $\mathcal M_{p_g,0,c_1^2}$;  the general points of $\mathcal M'$ correspond to surfaces of general type which can be canonically embedded and $\mathcal M'$ contains a codimension $1$ stratum $\mathcal S$  that parameterizes surfaces whose canonical map is a degree $2$ morphism.}
\end{remark}

\noindent Finally, in the next remark we compare the dimension of $\mathcal M'$ with the dimension of the moduli components $\mathcal M, \mathcal M_2$ and $\mathcal M_3$ of the surfaces of~\cite[Theorem 3.18]{canonical} and~\cite[Theorem 1.7]{Hirzebruch} having the same invariants as the surfaces of Proposition~\ref{construct.canonical.surfaces}. We recall that the components $\mathcal M, \mathcal M_2$ and $\mathcal M_3$ parameterize surfaces of general type whose canonical map is a degree $2$ morphism.

\begin{remark}\label{twocomponents}
{\rm In view of Remark~\ref{existence}, \cite[Theorem 4.11]{canonical}, \cite[Examples 3.7 and 3.8 and Remark 3.9]{Hirzebruch} and Proposition~\ref{construct.canonical.surfaces} show the following:
\begin{enumerate}
 \item For the triplets $(p_g,q,c_1^2)=(9, 0, 20), (15, 0, 38), (23, 0,  62), (33, 0, 92)$ %(39,0,110)$ and $(45,0,128)$,
there exist at least two components of $\mathcal M_{(p_g,q,c_1^2)}$. One of them is the component, which we  called $\mathcal M'$, described in Proposition~\ref{moduli}. The  other one, which we will call $\mathcal M$, parameterizes surfaces of general type as in~\cite[Theorem 3.18]{canonical}, therefore with a canonical map which is a degree $2$ morphism.
\item For the triplet $(39,0,110)$, there exist at least two components of $\mathcal M_{(p_g,q,c_1^2)}$. One of them is the component, which we called $\mathcal M'$, described in Proposition~\ref{moduli}. The  other one, which we will call ${\mathcal M_2}$, parameterizes surfaces of general type as in~\cite[Theorem 1.7]{Hirzebruch} (see also~\cite[Example 3.7]{Hirzebruch}), therefore with a canonical map which is a degree $2$ morphism.
\item For the triplet $(45,0,128)$, there exist at least three components of $\mathcal M_{(p_g,q,c_1^2)}$. One of them is the component, which we called $\mathcal M'$, described in Proposition~\ref{moduli}. The  other ones, which we will call ${\mathcal M_2}$ and ${\mathcal M_3}$, parameterize surfaces of general type as in~\cite[Theorem 1.7]{Hirzebruch} (see also~\cite[Example 3.8]{Hirzebruch}), therefore with a canonical map which is a degree $2$ morphism.
\end{enumerate}
The dimensions $\mu, \mu_2, \mu_3$ and $\mu'$ of $\mathcal M, \mathcal M_2, \mathcal M_3$ and $\mathcal M'$ respectively, follow from~\cite[Proposition 4.6]{canonical}, \cite[Proposition 3.6]{Hirzebruch} and Proposition~\ref{moduli} and are summarized in the following table:

\medskip

 \centerline{\vbox{\tabskip=0pt \offinterlineskip
%\tabskip= .25 truecm
\def\tablerule{\noalign{\hrule}}
\halign to %3.65truecm
6.5 truecm
%{\valign to125pt}
{\strut
#& \vrule#
\tabskip=0em plus 3em
%\tabskip= .25 truecm
& \hskip .25cm
\hfil #
\hfil  \hskip .05cm
& \vrule #
& \hskip .3cm
\hfil #
\hfil  \hskip .05cm
& \vrule \vrule \vrule #
&\hskip .25cm
\hfil# \hfil \hskip .05cm
&  \vrule #
& \hskip .25cm
\hfil #
\hfil  \hskip .05cm
& \vrule #
& \hskip .3cm
\hfil #
\hfil  \hskip .05cm
& \vrule #
& \hskip .3cm
\hfil #
\hfil  \hskip .05cm
& \vrule #&
\hskip .15cm \hfil# %\hskip .05truecm
\hfil & \vrule#\tabskip=0pt
\cr\tablerule
%&&Type
&&%\omit\hidewidth
$p_g$ %\hidewidth
&& %\omit\hidewidth
$c_1^2$ %\hidewidth
&& %\omit\hidewidth
$\mu$ %\hidewidth
&&%\omit\hidewidth
$\mu_2$ %\hidewidth
&&%\omit\hidewidth
$\mu_3$ %\hidewidth
&&%\omit\hidewidth
$\mu'$ %\hidewidth
&\cr\tablerule
%\cr
\tablerule
\tablerule
&&$9$
&&$20$
&&$63$
&&{\bf --}
&&{\bf --}
&&$63$
&\cr
\tablerule
%&&2
&&$15$
&&$38$
&&$96$
&&{\bf --}
&&{\bf --}
&&$93$
&\cr
\tablerule
%&&3
&&$23$ %\ *$ \hskip -.2truecm
&&$62$
&&$141$
&&{\bf --}
&&{\bf --}
&&$133$
&\cr
\tablerule
%&&4
&&$33$ % \ *$ \hskip -.2truecm
&&$92$
&&$198$
&&{\bf --}
&&{\bf --}
&&$183$
&\cr
\tablerule
%&&4
&&$39$ % \ *$ \hskip -.2truecm
&&$110$
&&{\bf --}
&&$233$
&&{\bf --}
&&$213$
&\cr
\tablerule
&&$45$ % \ *$ \hskip -.2truecm
&&$128$
&&{\bf --}
&&$267$
&&$266$
&&$243$
&\cr
\tablerule
\noalign{\smallskip}}}}
}
\end{remark}

\begin{acknowledgement}
{\rm We are grateful to Edoardo Sernesi for helpful conversations. We also thank Tadashi Ashikaga for bringing to our attention his result~\cite[4.5]{AK} with Kazuhiro Konno.}
\end{acknowledgement}

\end{document}